\documentclass[twoside,a4paper,12pt,centertags]{amsart}
\usepackage{amsmath,amssymb,verbatim,vmargin}
\usepackage[bookmarks=true]{hyperref}   % Hyperref in DVI and PDF
\usepackage{color}

\theoremstyle{plain}
\newtheorem{thm}{Theorem}[section]
\newtheorem{lem}[thm]{Lemma}
\newtheorem{cor}[thm]{Corollary}
\newtheorem{prop}[thm]{Proposition}

\theoremstyle{definition}
\newtheorem{defn}[thm]{Definition}

\mathchardef\semic="303B
\newcommand{\R}{{\mathbf R}}
\newcommand{\C}{{\mathbf C}}
\newcommand{\Z}{{\mathbf Z}}

\newcommand{\mL}{{\mathcal L}}

\newcommand{\mD}{{\mathcal D}}

\newcommand{\sett}[2]{ \{ #1 \, \semic \, #2 \} }

\newcommand{\barint}{\mbox{$ave \int$}}
\newcommand{\divv}{{\text{{\rm div}}}}

\newcommand{\bmo}{\text{BMO}}

\newcommand{\tr}{\text{{\rm Tr}}\,}

\def\barint_#1{\mathchoice
            {\mathop{\vrule width 6pt
height 3 pt depth -2.5pt
                    \kern -8.8pt
\intop}\nolimits_{#1}}%
            {\mathop{\vrule width 5pt height
3 pt depth -2.6pt
                    \kern -6.5pt
\intop}\nolimits_{#1}}%
            {\mathop{\vrule width 5pt height
3 pt depth -2.6pt
                    \kern -6pt
\intop}\nolimits_{#1}}%
            {\mathop{\vrule width 5pt height
3 pt depth -2.6pt
          \kern -6pt \intop}\nolimits_{#1}}}

\usepackage{color}

\definecolor{gr}{rgb}   {0.,   0.8,   0. }
\definecolor{bl}{rgb}   {0.,   0.5,   1. }
\definecolor{mg}{rgb}   {0.7,  0.,    0.7}

\makeatletter
\@namedef{subjclassname@2010}{%
  \textup{2010} Mathematics Subject Classification}
\makeatother

\begin{document}

\title[A local $Tb$ theorem for matrix weighted paraproducts]
{A local $Tb$ theorem for matrix weighted paraproducts}
\author[Andreas Ros\'en]{Andreas Ros\'en$\,^1$}
\thanks{$^1\,$Formerly Andreas Axelsson}
\address{Andreas Ros\'en\\Mathematical Sciences, Chalmers University of Technology and University of Gothenburg\\
SE-412 96 G{\"o}teborg, Sweden}
\email{andreas.rosen@chalmers.se}

\begin{abstract}
  We prove a local $Tb$ theorem for paraproducts acting on vector valued functions,
  with matrix weighted averaging operators. 
  The condition on the weight is that its square is in the $L_2$ associated matrix $A_\infty$ class.
  We also introduce and use a new matrix reverse H\"older class.
  This result generalizes the previously known case of scalar weights from the proof of the Kato square root problem, as well as the case of diagonal weights, recently used in the study of boundary value problems for degenerate elliptic equations.
\end{abstract}

\thanks{The author was supported by Grant 621-2011-3744 from the Swedish research council, VR}
\keywords{local $Tb$ theorem, paraproduct, matrix weight, stopping time argument, Carleson measure}
\subjclass[2010]{42B20, 42B37, 42C99.}

\maketitle

% 42B25 (1980-now) Maximal functions, Littlewood-Paley theory
% 42B35 (2000-now) Function spaces arising in harmonic analysis
% 42B37 (2010-now) Harmonic analysis and PDE

%
%
%
\section{Introduction}

Paraproducts first appeared in the work of J.-M. Bony~\cite{Bony} on paradifferential operators.
They are closely related to Carleson measures, see \cite{Carleson:58} and \cite{Carleson:62} for the original formulations, and are fundamental in harmonic analysis, in particular in the study of singular integral operators and square functions.
One way to introduce paraproducts is by examining a multiplication operator $f(x)\mapsto b(x)f(x)$
in a wavelet basis. For this introduction, consider the standard Haar basis $\{h_Q\}_{Q\in \mD}$ for $L_2(\R)$, where $Q$ runs through the dyadic intervals $\mD$.
Then we have the $L_2$ identity
\begin{equation}  \label{eq:paraproddec}
  bf= \sum_{Q,R} \Delta_Q b \,\Delta_R f=
  \sum_Q E_Q b\, \Delta_Q f+ \sum_Q \Delta_Q b\, E_Q f + \sum_Q \Delta_Q b\, \Delta_Q f,
\end{equation}
where $\Delta_Q g:= (g,h_Q)h_Q$ and
$E_Q g$ denotes the average of $g$ on $Q$.
The identity follows by splitting the double sum into the three cases $Q\supsetneqq R$,
$Q\subsetneqq R$ and $Q=R$, and using the relation $\sum_{R\supsetneqq Q} \Delta_R= E_Q$ on $Q$ for the
first two terms.
The first term represents a Haar multiplier
$\pi^0_bf:= \sum_Q E_Q b\, \Delta_Q f$ with a rather trivial estimate
$\|\pi^0_b\|_{L_2\to L_2}= \|b\|_{L_\infty}$.
The second term defines a paraproduct operator
$$
 \pi^+_bf:= \sum_Q \Delta_Q b\, E_Q f.
$$
This operator enjoys the sharper bound
\begin{equation}  \label{eq:paraestsimple}
  \|\pi^+_bf\|_{L_2}^2= \sum_Q |(b,h_Q)|^2 |E_Q f|^2 \lesssim \|b\|_{\bmo}^2 \|f\|_{L_2}^2,
\end{equation}
as seen by applying Carleson's embedding theorem and
$\sum_{R\subset Q}|(b,h_Q)|^2 = \int_Q |f(x)-E_Qf|^2 dx$.
Here $\bmo(\R)$ is the John--Nirenberg space of functions of bounded mean oscillation.
Before continuing, we also mention that the third term in \eqref{eq:paraproddec}
turns out to be the adjoint paraproduct $(\pi^+_b)^*f$.

Paraproducts like $\pi^+_b$ are fundamental objects in modern harmonic analysis.
For example they are found at the heart of singular integral theory.
Given a singular integral operator
$$
  T f(x)= \int_\R k(x,y) f(y) dy, \qquad x\in \R,
$$
one can show that
$$
  T= \pi^0_m + \pi^+_{T(1)}+ (\pi^+_{T^*(1)})^*+ R.
$$
See Auscher, Hofmann, Muscalu, Tao and Thiele~\cite{AHMTT} and Auscher and Yang~\cite{AY}.
Here the Haar multiplier is controlled by a weak boundedness hypothesis on $T$ and the
$L_2$ norm of the remainder term $R$ can be estimated using a H\"ormander regularity condition
on the kernel. What remains for $L_2$ boundedness of $T$ is the famous $T(1)$ condition:
The functions $T(1)$ and $T^*(1)$ should belong to $\bmo(\R)$.

Also at the cross roads of operator theory and harmonic analysis, paraproducts are fundamental.
At the core of the proof of the Kato square root estimate
$$
\|\sqrt{-\divv A\nabla}u\|_{L_2(\R^n)}\approx\|\nabla u\|_{L_2(\R^n)}
$$
for uniformly elliptic operators,
proved by Auscher, Hofmann, Lacey, McIntosh and Tchamitchian~\cite{AHLMcT}, there
is a paraproduct, or Carleson, estimate
\begin{equation}   \label{eq:paraprodKato}
  \int_0^\infty\int_{\R^n} |\gamma(t,x) E_t f(x)|^2 dx\frac{dt}t\lesssim \|f\|_{L_2(\R^n)}^2,
\end{equation}
similar to \eqref{eq:paraestsimple}.
We have here passed to a continuous setting, and $E_t$ denotes some standard mollifier at scale $t$.
More importantly, for the Kato square root estimate one needs to work with vector-valued functions
$f:\R^n\to \C^N$, on which $E_t$ act componentwise, and the Carleson multiplier $\gamma_t(x)$ is matrix
valued.
To establish such a paraproduct estimate, one needs to prove that $|\gamma_t(x)|^2 dxdt/t$ is a Carleson
measure. This is proved by a local $T(b)$ argument, a technique pioneered by Christ~\cite{Chr},
motivated by applications to analytic capacity.
The basic idea in such local $T(b)$ arguments is that we suppose given
a family of test functions $b_Q$, one for each dyadic cube, adapted to the object $\gamma(t,x)$
that we seek to estimate. More concretely, one aims to establish the Carleson bound
of $\gamma(t,x)$ by reducing it to an estimate of $\gamma(t,x) E_t b_Q(x)$,
assuming only a lower bound on the mean of $b_Q$ and not a pointwise lower bound.
In the vector-valued situation, an even more refined localisation is needed: A test function $b_Q^v$
for each cube $Q$ and each unit vector $v$ is required such that the mean of $b_Q^v$ on $Q$ is in the
direction specified by $v$.

More recently, the need arose for a generalization to the weighted setting of the local $T(b)$ theorem for
\eqref{eq:paraprodKato}, in connection
with work on boundary value problems for degenerate elliptic equations by
Auscher, Ros\'en and Rule~\cite{ARR}.
The case of scalar weights is a straightforward generalization of the methods from \cite{AHLMcT}.
The argument in \cite{ARR} required certain diagonal matrix-valued weights,
which was a non-trivial extension of \cite{AHLMcT}.
The following main result in this paper concerns the full matrix weighted situation,
which is a non-trivial extension of the diagonal case from \cite{ARR}.

\begin{thm}   \label{thm:main}
  Fix positive integers $n, N, M\ge 1$.
  Let $d\mu= \mu(x)dx$ be a doubling measure on $\R^n$, say $\mu(2Q)\le C_1 \mu(Q)$ for all cubes
  $Q\subset \R^n$ with concentric double $2Q$, $C_1<\infty$, and consider
  the function space $L_2(\mu)= L_2(\R^n, \mu; \R^N)$.
  Let $W: \R^n\to\mL(\R^N)$ be a matrix-valued function such that $W(x)$ is a positive definite
  matrix for almost every $x\in \R^n$, and satisfies the estimate
\begin{equation}  \label{eq:theWest}
  \det \left(\barint_Q W(x)^2 d\mu(x)\right)\le C_2^2
  \exp\left( \barint_Q \ln (\det W(x)^2) d\mu(x) \right),
\end{equation}
for some constant $C_2<\infty$, uniformly for all cubes $Q\subset\R^n$.
Define the dyadic weighted averaging/expectation operator $E_t f(x):= E_Q f$,
where $E_Qf$ denotes the matrix weighted average
$$
  E_Q f:= \left( \int_Q W(x) d\mu(x) \right)^{-1} \int_Q W(x)f(x) d\mu(x)
$$
over the dyadic cube $Q\ni x$ of sidelength $\ell(Q)= 2^k$, $(2^{k-1},2^k]\ni t$.

Let $(t,x)\mapsto \gamma(t,x)\in \mL(\R^N; \R^M)$ be a given matrix-valued function in the half space $\R^{1+n}_+$.
Then we have the Carleson estimate
$$
  \sup_{Q\in\mD} \frac 1{\mu(Q)}\iint_{\R^{1+n}_+} |\gamma(t,x)|^2 \frac{dtd\mu(x)}t <\infty,
$$
provided the following holds.
There should exist a family of test functions $b_Q^v\in L_2(\mu)$,
one for each dyadic cube $Q$ and for each unit vector $v\in \R^N$, such that
\begin{gather*}
  (v, E_Qb_Q^v)= 1, \\
  \int_{\R^n} |b_Q^v(x)|^2 d\mu(x)\le C_3^2 \mu(Q), \\
  \iint_{\R^{1+n}_+}|\gamma(t,x) E_t b_Q^v(x) |^2  \frac{dtd\mu(x)}t \le C_4^2 \mu(Q),
\end{gather*}
for some $C_3, C_4<\infty$.
\end{thm}

We note that it suffices to prove the theorem for $M=1$ since the hypothesis give 
estimates for each row in $\gamma(t,x)$, but we shall not need to make this reduction.
Also, in the two estimates involving $\gamma(t,x)$, it is clear from the proof that 
we can replace $|\gamma(t,x)|^2$
by some more general non-negative radial function $\phi(|\gamma(t,x)|)$, assuming
estimates of $\phi(|\gamma(t,x) E_t b_Q^v(x)|)$.
We also remark that the proof is purely real-variable, and the result applies directly
to complex valued functions, simply by writing $\C^N= \R^{2N}$.
Note that unless the operators are assumed to be complex linear, we need here, even
in the case $\R^N=\C$, different test functions $b_Q^v$ for different directions $v$.

Previously known special cases of Theorem~\ref{thm:main} are the following.

\begin{itemize}
  \item [{\rm (i)}]
    The unweighted case $W=I$ and $d\mu= dx$ is the core part of the proof of the Kato square
    root estimate by Auscher, Hofmann, Lacey, McIntosh and Tchamitchian~\cite{AHLMcT}.
  \item  [{\rm (ii)}]
    The diagonal matrix weighted case $N= 1+n$ and $W= \begin{bmatrix} 1 & 0 \\ 0 & \mu^{-1}I \end{bmatrix}$
    was needed and proved by the author jointly with Auscher and Rule in \cite{ARR} as mentioned above.
\end{itemize}

A straightforward generalization of (ii) is that of diagonal matrix weights
$$
  W= \begin{bmatrix} w_1 & \cdots & 0 \\ \vdots & \ddots & \vdots \\ 0 & \cdots & w_N \end{bmatrix}.
$$
In this case, the estimate \eqref{eq:theWest} reduces to the scalar $A_\infty(d\mu)$ estimate for
each of $w_1^2,\ldots, w_N^2$. This is known to be equivalent to the reverse H\"older estimate
$B_2(d\mu)$ for each of $w_1,\ldots, w_N$,
that is
$$
  \left(\barint_Q w_i^2 d\mu\right)^{1/2}\lesssim \barint_Q w_id\mu, \qquad i=1,\ldots, N.
$$
See Coifman and Fefferman~\cite{CF} and Buckley~\cite[Prop. 3.10]{Buck}.
In the case (ii) this amounts to the Muckenhoupt condition $\mu\in A_2(dx)$.

Turning to the meaning of \eqref{eq:theWest} in the general matrix case, we have a sequence of
estimates
\begin{multline*}
  \left(\det \barint_Q W^2 d\mu \right)^{1/2}\ge \det \barint_Q W d\mu
  \ge \exp\left( \barint_Q \ln(\det W) d\mu \right) \\
  \ge \left(\det \barint_Q W^{-1} d\mu \right)^{-1}\ge \left(\det \barint_Q W^{-2} d\mu \right)^{-1/2},
\end{multline*}
which hold for any $Q$ and any matrix weight $W$, that is measurable positive matrix-valued function.
\begin{itemize}
  \item [{\rm (i)}]
  A reverse estimate $\left(\det \barint_Q W^2 d\mu \right)^{1/2}\lesssim \det \barint_Q W d\mu$, uniformly for $Q$,
amounts to a matrix reverse H\"older condition $B_2^N(d\mu)$, which we have not found in the literature in the non-scalar case $N\ge 2$.
  \item [{\rm (ii)}]
  A reverse estimate $\det \barint_Q W d\mu
  \lesssim \exp\left( \barint_Q \ln(\det W) d\mu \right)$, uniformly for $Q$,
amounts to a matrix $A_\infty^N(d\mu)$ condition.
See Volberg~\cite{Vol}, where this appears as the $A_{2,\infty}$ condition.
  \item [{\rm (iii)}]
  A reverse estimate $\det \barint_Q W d\mu
  \lesssim \left(\det \barint_Q W^{-1} d\mu \right)^{-1}$, uniformly for $Q$,
amounts to a matrix $A_2^N(d\mu)$ condition.
See Treil and Volberg~\cite{TV}, where this matrix Muckenhoupt $A_2$ condition appears.
\end{itemize}

We note that \eqref{eq:theWest} amounts to assuming both the matrix $B_2^N(d\mu)$ and $A_\infty^N(d\mu)$
conditions. In the non-scalar case, we do not know if $B_2^N(d\mu)$ implies $A_\infty^N(d\mu)$, as it does in the scalar case.

The above conditions on matrix weights used in this paper are all related to $L_2(d\mu)$.
It is known, see \cite{Vol}, that conditions on matrix weights for $L_p$ estimates, $p\ne 2$,
cannot be formulated in terms of some $L_q$ averages of weights and their determinants.
Instead the appropriate point of view is to consider norm-valued weights, rather than
matrix-valued weights.

The proof of Theorem~\ref{thm:main} is in Section~\ref{sec:proof}. Similar to \cite{ARR} it combines 
two parallel stopping time arguments, one for the test functions and one for the weight.
The one for the test functions, Proposition~\ref{prop:Katostop}, follows \cite{AHLMcT}, but just like 
in \cite{ARR}, a factor depending on the weight remains to be controlled.
The stopping time argument used to control the weight in \cite{ARR}, uses in an essential way the scalar
nature of the weight. In this paper we find in Proposition~\ref{prop:coronastop} a stopping time argument for the weight that is simpler than that in \cite{ARR}, and which also generalize to matrix weights.
Proposition~\ref{prop:coronastop} is the main new technique of this paper, and it uses a stopped and matrix
version of the Carleson condition for scalar $A_\infty$ weights. See Fefferman, Kenig and Pipher~\cite[Sec. 3]{FKP} and Buckley~\cite[Sec. 5]{Buck}.

Proposition~\ref{prop:coronastop} for full matrix weights seems to be impossible to derive from
properties of scalar weights. An interesting reverse reverse triangle inequality, originating from a failed such attempt, is included in Section~\ref{sec:reverev}, since it seems interesting in its own right.
A reason why such reduction to scalar weights does not work, is that it seems to require infinitely many
parallel stopping time arguments as in Lemma~\ref{lem:JN}.
The argument in \cite{ARR} used that the logarithm of a Muckenhoupt weight is of bounded mean oscillation.
Although a matrix version of this holds true, see Bownik~\cite[Sec. 6]{Bow}, it seems impossible to 
prove Proposition~\ref{prop:coronastop} following \cite{ARR}, since the values of a matrix weight do not commute.

Going back to \cite{AHLMcT}, the proof of Theorem~\ref{thm:main} also requires a compactness argument
to handle the matrix valued multipliers $\gamma_t(x)$. To this end, it is standard to decompose
the space of matrices $\mL(\R^N)$ into small cones.
This argument would work in this paper as well, but we choose to make use of a somewhat sharper argument, 
Proposition~\ref{prop:Katonewsectors}, which instead use a conical decomposition of $\R^N$ itself.
Independently, this argument has earlier been found by Hyt\"onen~\cite{Hunpub}.

In Section~\ref{sec:matrixweights}, we discuss the new matrix reverse H\"older classes $B_2^N(d\mu)$ and 
the matrix $A_\infty^N(d\mu)$ class from \cite{Vol}.
In particular we note a matrix martingale square function estimate for $B_2^N(d\mu)$ weights
and a packing condition for $A_\infty^N(d\mu)$ weights, which are needed for the proof of Proposition~\ref{prop:coronastop}.

\section{Stopping criteria and sawtooths}    \label{sec:prel}

We denote cubes in $\R^n$ by $Q, R, S$.
In particular we shall consider the standard dyadic cubes $\mD=\bigcup_{t>0}\mD_t$,
where $\mD_t$ denotes the set of dyadic cubes of sidelength $\ell(Q)=2^{-k}$, where $k\in \Z$, $2^{-k-1}<t\le 2^{-k}$.

Throughout this paper, $\mu$ denotes a fixed doubling measure on $\R^n$.
We write $2Q$ for the cube concentric with $Q$ and having twice the
sidelength: $\ell(2Q)= 2\ell(Q)$.
We write $\mu(Q):= \int_Q d\mu$ for the $\mu$-measure of $Q$.
For a function $W$, possibly matrix valued, we write $W_Q:= \barint_Q W d\mu= \mu(Q)^{-1}\int_Q Wd\mu$,
for the $\mu$-average of $W$ over $Q$.

In this section we consider the main technique used in this paper: stopping time arguments.
The structure of these arguments is as follows.
We have a stopping criterion (S) which selects, for a given dyadic cube $Q\in\mD$, a
set $B_1(Q)$ of disjoint dyadic subcubes $R\in\mD$ of $Q$. By applying (S) to each of
these subcubes $R$, we obtain the set of second generation stopping cubes
$
  B_2(Q):= \bigcup_{R\in B_1(Q)} B_1(R)
$
under $Q$. Proceeding recursively, we define the $k$'th generation stopping cubes
$B_k(Q):= \bigcup_{R\in B_{k-1}(Q)} B_1(R)$, $k=3,4,\ldots$, and also set $B_0(Q):= \{Q\}$.
Let $B_*(Q):= \bigcup_{k=0}^\infty B_k(Q)$.
In this way we obtain from (S) a partition
$$
  \widehat Q= \bigcup_{S\in B_*(Q)} G(S)
$$
of the Carleson box $\widehat Q:= (0,\ell(Q))\times Q\subset \R^{1+n}$ into sawtooth regions
$$
G(S):= \widehat S\setminus \bigcup_{R\in B_1(S)} \widehat R.
$$
We allow us some abuse of notation: Sometimes we also consider the Carleson boxes as
sets of dyadic cubes $\widehat Q=\sett{R\in\mD}{R\subset Q}$, in which case the sawtooth regions
$G(S)$ also are viewed as sets of dyadic cubes.
This is possible due to the one-to-one correspondence between dyadic cubes $Q\subset\R^n$
and dyadic Whitney regions $(\ell(Q)/2,\ell(Q))\times Q\subset\R^{1+n}_+$.

For a given collection of dyadic cubes $\widetilde \mD\subset \mD$, for example
$\widetilde \mD= B_*(Q)$, we denote by $R_*$ the stopping parent of $R\in \mD$, that is
the smallest cube $R_*\in \widetilde D$ such that $R_*\supset R$.

Two basic techniques are the following.

\begin{lem}   \label{lem:JN}
  Consider a finite number of stopping criteria, producing stopping cubes
  $B_*^1(Q)$, \ldots, $B_*^k(Q)$.
  Assume that we have packing conditions
  $$
  \sup_Q \frac 1{\mu(Q)}\sum_{R\in B^i_*(Q)} \mu(R)<\infty, \qquad i=1,\ldots, k.
  $$
  Then the Carleson estimate $\sup_Q\frac 1{\mu(Q)}\iint_{\widehat Q} f(t,x)dtdx <\infty$
  follows from
  $$
    \sup_{S_1\supset\ldots\supset S_k}\frac 1{\mu(S_k)}\iint_{G^1(S_1)\cap G^2(S_2)\cap\ldots\cap G^k(S_k)} f(t,x)dtdx<\infty,
  $$
  for any given measureable function $f(t,x)\ge 0$ on $\R^{1+n}_+$.
\end{lem}

\begin{proof}
Consider the iterated decomposition into sawtooth regions
$$
  \widehat Q= \bigcup_{S_1\in B_*^1(Q)} \bigcup_{S_2\in B_*^2(S_1)}\ldots \bigcup_{S_k\in B_*^{k}(S_{k-1})}
  G^1(S_1)\cap \ldots\cap G^k(S_k).
$$
Use the given estimate and then sum, using the given $k$ packing conditions.
This proves the Carleson estimate.
\end{proof}

For some types of stopping criteria, one can verify the packing condition in
Lemma~\ref{lem:JN} as follows.

\begin{lem}   \label{lem:geomet}
  If $\sup_Q \frac 1{\mu(Q)}\sum_{R\in B_1(Q)} \mu(R)<1$, then
  $\sup_Q \frac 1{\mu(Q)}\sum_{R\in B_*(Q)} \mu(R)<\infty$.
\end{lem}

\begin{proof}
  For some $c<1$ and all $S\in \mD$, we have
  $\sum_{R\in B_1(S)} \mu(R)\le c\mu(S)$. Summing over $S\in B_k(Q)$, we obtain
  $$
    \sum_{R\in B_{k+1}(Q)} \mu(R)= \sum_{S\in B_k(Q)}\sum_{R\in B_1(S)} \mu(R)\le c \sum_{S\in B_k(Q)}\mu(S)\le\ldots \le c^{k+1}\mu(Q).
  $$
  Finally we can sum over $k$, since the geometric series converges.
\end{proof}

\section{Matrix and scalar weights}   \label{sec:matrixweights}

In this section we collect the estimates that we need for matrix weights, that is functions $W:\R^n\to \mL(\R^N)$
for which $W(x)$ is a positive definite matrix for almost all $x\in\R^n$.
We also recall some needed well known results for scalar weights, that is in the case $N=1$.

\begin{prop}   \label{prop:mB2}
For a given matrix weight $W$ and cube $Q\subset \R^n$, consider the following estimates.
\begin{itemize}
  \item [{\rm (i)}]
  $\left( \barint_Q |W(x)a|^2d\mu(x) \right)^{1/2}\lesssim \left| \barint_Q W(x)ad\mu(x) \right|$,
   for all $a\in\R^N$
  \item [{\rm (ii)}]
  $\left|\left( \barint_Q W(x)^2d\mu(x) \right)^{1/2}\left( \barint_Q W(x)d\mu(x) \right)^{-1}a\right|\lesssim |a|$,
   for all $a\in\R^N$
    \item [{\rm (iii)}]
  $\left|\left( \barint_Q W(x)d\mu(x) \right)^{-1}\left( \barint_Q W(x)^2d\mu(x) \right)\left( \barint_Q W(x)d\mu(x) \right)^{-1}a\right|\lesssim |a|$,
   for all $a\in\R^N$
    \item [{\rm (iv)}]
  $\det\left( \barint_Q W(x)^2d\mu(x) \right)^{1/2}\lesssim \det\left( \barint_Q W(x)d\mu(x) \right)$
\end{itemize}
If one of these estimates holds, then the other three also hold.
The four estimates $\ge$ always hold with constant $1$.
\end{prop}

The proof below is analogous to the corresponding proof for matrix $A_2$ weights.
See Lauzon and Treil~\cite[Sec. 3.1.2]{LT} and references therein.

\begin{proof}
We note that
$$
  \left( \barint_Q |W(x)a|^2d\mu(x) \right)^{1/2}= \left|\left( \barint_Q W(x)^2d\mu(x) \right)^{1/2}a \right|
$$
and $|A|^2= |A^*A|$, where
$$
 A:=\left( \barint_Q W(x)^2d\mu(x) \right)^{1/2}\left( \barint_Q W(x)d\mu(x) \right)^{-1}.
$$
From this the equivalence of (i)-(iii) is clear.
The inequality $\ge$ is clear for (i), and since $A$ and $A^*A$ are inverses of contractions,
the inequality $\ge$ follows for (ii) and (iii).
Again noting that $|A^{-1}|\le 1$, the equivalence with estimate (iv) is clear.
\end{proof}

\begin{prop}   \label{prop:mAinfty}
For a given matrix weight $W$ and cube $Q\subset \R^n$, consider the following estimates.
\begin{itemize}
  \item [{\rm (i)}]
    $\exp\left( \barint_Q \ln |W(x)^{-1/2} a| d\mu(x) \right)\lesssim
   \left|\left(\barint_Q W(x) d\mu(x)\right)^{-1/2}a \right|$, for all $a\in\R^N$
  \item [{\rm (ii)}]
    $\det\left( \barint_Q W(x)d\mu(x) \right)\lesssim \exp\left( \barint_Q \ln\det W(x)d\mu(x) \right)$
\end{itemize}
If one of these estimates holds, then the other also holds.
The two estimates $\ge$ always hold with constant $1$.
\end{prop}

This is proved in \cite[Sec. 2]{Vol}. For completeness we sketch the proof.

\begin{proof}
To prove $\ge$ in (i), by duality it suffices to prove an upper bound on
$|W_Q^{1/2}a|$, which follows from Jensen's inequality and the estimate
$|(a,b)|\le |W(x)^{-1/2}a| |W(x)^{1/2}b|$.
The proof of $\ge$ for (ii) uses the determinant version of the geometric-arithmetic mean
inequality. See \cite[Lem. 4.2]{TV} for the discrete version of this.

For the equivalence of estimates $\lesssim$, we write (i) as $\exp\left( \barint_Q \ln |T(x) a| d\mu(x) \right)\lesssim |a|$ and (ii) as $\exp\left( \barint_Q \ln\det T(x)d\mu(x) \right)\lesssim 1$,
where $T=T(x):= W(x)^{-1/2}W_Q^{1/2}$.
The proof of (i)$\Rightarrow$(ii) follows from $\det T\le \prod_i |Te_i|$, with $\{e_i\}$ being an ON-basis.
For the proof of (ii)$\Rightarrow$(i), let $b(x)$ be a unit vector which minimizes
$|T(x)b(x)|$, so that $|T(x) a||T(x)b(x)|^{N-1}\le \det T(x)$ for any unit vector $a$.
Using $\ln^+(t)= \ln(t)+ \ln^+(1/t)$, where $\ln^+(t):= \max(\ln(t),0)$, with
$t= |T(x)a|$ and $t=|T(x)b(x)|^{N-1}$ respectively,
we obtain
\begin{multline*}
  \barint_Q \ln^+|T(x)a| d\mu
  \\
  \le
  \barint_Q\ln \det T(x)d\mu+ \barint_Q\ln^+(|T(x) a||T(x)b(x)|^{N-1})^{-1} d\mu  \\
  \lesssim 1+\barint_Q\ln^+|T(x)^{-1}|d\mu
  \lesssim 1+\barint_Q |T(x)^{-1}|^2 d\mu  \\
  = 1+\barint_Q |W(x)^{1/2}W_Q^{-1/2}|^2 d\mu
  \lesssim 1.
\end{multline*}
\end{proof}

\begin{defn}
  Let $W: \R^n\to \mL(\R^N)$ be a matrix function, that is a measurable function for which
  $W(x)$ is a positive definite matrix for almost every $x\in\R^n$.
  We write $W\in B_2^N(d\mu)$ if the four equivalent estimates in Proposition~\ref{prop:mB2}
  hold, uniformly for all (possibly non-dyadic) cubes $Q$.

  We write $W\in A_\infty^N(d\mu)$ if the two equivalent estimates in Proposition~\ref{prop:mAinfty}
  hold, uniformly for all (possibly non-dyadic) cubes $Q$.

  In the case of scalar weights, $N=1$, we write $B_2(d\mu):= B_2^1(d\mu)$
  and $A_\infty(d\mu):= A_\infty^1(d\mu)$.
\end{defn}

A first use of this matrix reverse H\"older condition is the following estimate of 
matrix weighted averages. Further related estimates of square functions and maximal functions
are also possible, we shall not need that in this paper.

\begin{prop}   \label{prop:Etbound}
  The weighted averaging operators $E_t$ in Theorem~\ref{thm:main}
  have estimates
  $$
    \sup_{t>0}\|E_t\|_{L_2(\mu)\to L_2(\mu)}<\infty.
  $$
\end{prop}

\begin{proof}
  Consider a dyadic cube $Q$.
  Let $v\in \R^N$ be a unit vector such that
  $|E_Q f|= (E_Qf, v)$.
  Using the matrix reverse H\"older inequality from Proposition~\ref{prop:mB2}, we obtain
\begin{multline*}
  |E_Q f|= \int_Q\left( W(x) (\mu(Q)W_Q)^{-1}v, f(x) \right) d\mu(x)\\
  \le \left( \int_Q | W(x) (\mu(Q)W_Q)^{-1} v|^2 d\mu(x)\right)^{1/2}
  \left( \int_Q |f(x)|^2 d\mu(x)\right)^{1/2} \\
  \lesssim \mu(Q)^{-1/2} \left| \int_Q W(x)(\mu(Q)W_Q)^{-1}v d\mu(x)\right|
  \left( \int_Q |f(x)|^2 d\mu(x)\right)^{1/2} \\
  = \mu(Q)^{-1/2} \left( \int_Q |f(x)|^2 d\mu(x)\right)^{1/2}.
\end{multline*}
Summing over $Q\in \mD_t$ gives the stated estimate.
\end{proof}

The following two propositions contain the main estimates for matrix weights that we need.
The first is a martingale square function estimate, which is a stopped and matrix-valued version of
\cite[VIII, Lem. 6.4]{Gar}.

\begin{prop}   \label{prop:matrixmarting}
Let $B_*(Q)$ denote the generations of stopping cubes under $Q\in \mD$ obtained from
 some stopping criterion, and let W be a matrix weight with $W\in L_2(Q;\mL(\R^N))$.
 Then, as positive matrices, we have the estimate
$$
  \sum_{R\in B_*(Q)\setminus \{Q\}} (W_R-W_{R_*})^2 \mu(R) \le ((W^2)_Q-(W_Q)^2 )\mu(Q).
$$
\end{prop}

\begin{proof}
Define functions $f_k:Q\to \mL(\R^N)$, $k=0,1,2,\ldots$, by
$$
  f_k(x):= \begin{cases} W_R, & x\in R\in B_k(Q), \\
  W(x), & x\in Q\setminus \bigcup B_k(Q).
  \end{cases}
$$
We note the integral identities
$$
  \int_{Q} f_{k+1}(x)f_k(x)d\mu(x)= \int_{Q} f_{k}(x)f_{k+1}(x)d\mu(x)=\int_{\R^n} f_k(x)^2 d\mu(x).
$$
This gives the telescoping sum
\begin{multline*}
  \sum_{R\in B_*(Q)\setminus \{Q\}} (W_R-W_{R_*})^2 \mu(R)
  \le \sum_{k=0}^\infty\int_Q (f_{k+1}(x)-f_k(x))^2 d\mu(x)\\
  = \sum_{k=0}^\infty \int_Q (f_{k+1}(x)^2-f_k(x)^2) d\mu(x) = \lim_{k\to\infty} \int_Q( f_k(x)^2-f_0(x)^2) d\mu(x) \\ = ((W^2)_Q-(W_Q)^2) \mu(Q).
\end{multline*}
\end{proof}

\begin{prop}   \label{prop:volbergstop}
  Assume that $W^2\in A_\infty^N(d\mu)$. Consider the stopping criterion that
  selects the maximal dyadic subcubes $R$ of $Q$ for which
  $$
    |W_QW_R^{-1}|\ge \lambda.
  $$
  Then for the first generation stopping cubes we have the
  packing condition
  $$
    \sum_{R\in B_1(Q)}\mu(Q)\lesssim (\ln\lambda)^{-1}\mu(Q).
  $$
   for $\lambda$ large enough.
\end{prop}

The analogous result for $W\in A_\infty^N(d\mu)$ is proved in \cite[Lem. 3.1]{Vol}.
The following proof for $W^2$ is a straightforward modification, but we give
it for completeness.

\begin{proof}
  By combining Propositions~\ref{prop:mAinfty} and \ref{prop:mB2},
  we obtain
  $$
     \exp\left( \barint_Q \ln |W(x)^{-1} a| d\mu(x) \right)\approx
   |W_Q^{-1} a|, \qquad \text{for all } a\in\R^N.
  $$
  Fix a unit vector $a$ and write $B^a(Q)$ for the sets of maximal dyadic subcubes $R$ of $Q$
  for which $|W_R^{-1}W_Q a|\ge \lambda/\sqrt n$. For $R\in B^a(Q)$, we have
  $$
  \exp\left( \barint_R \ln |W(x)^{-1} W_Q a| d\mu(x) \right)\approx |W_R^{-1}W_Q a|\ge \lambda/\sqrt n.
  $$
  With estimates similar to the proof of Proposition~\ref{prop:mAinfty}, this gives
  \begin{multline*}
    \ln(\lambda/\sqrt n) \sum_{R\in B^a(Q)}\mu(R)\lesssim \sum_{R\in B^a(Q)}
    \int_R \ln |W(x)^{-1} W_Q a| d\mu(x) \\
    \lesssim
    \int_Q \ln |W(x)^{-1} W_Q a| d\mu(x) + \int_Q \ln^+(|W(x)^{-1} W_Q a|)^{-1} d\mu(x)\\
     \lesssim \mu(Q) + \int_Q \ln^+|W(x) W_Q^{-1}| d\mu(x) \\
     \lesssim
   \mu(Q)+ \int_Q |W(x) W_Q^{-1}|^2 d\mu(x)\lesssim \mu(Q).
  \end{multline*}
  This completes the proof since $B_1(Q)\subset \bigcup_i B^{e_i}(Q)$, if $\{e_i\}$ is an
  ON-basis.
\end{proof}

We also need the following classical result for scalar weights.

\begin{prop}   \label{prop:scalarAinfty}
  Assume that $d\mu$ is a doubling measure on $\R^n$.
  Consider a scalar weight $w$ on $\R^n$ and define the measure $d\sigma:= w d\mu$.
  Then the following are equivalent.
  \begin{itemize}
    \item [{\rm (i)}]
      There exists $p>1$ so that uniformly for all cubes $Q$, we have
      the $A_{p'}$ estimate 
      $$
      \barint_Q w d\mu\lesssim \left( \barint_Q w^{-(p-1)} d\mu\right)^{-1/(p-1)}.
      $$
    \item [{\rm (ii)}]
      Uniformly for all cubes $Q$, we have
      the $A_\infty$ estimate 
      $$
      \barint_Q w d\mu\lesssim \exp\left( \barint_Q \ln w d\mu \right).
      $$
     \item [{\rm (iii)}]
      There exists $\alpha, \beta\in (0,1)$ so that uniformly for all cubes $Q$ and all subsets $E\subset Q$, we have $\sigma(E)\le \alpha \sigma(Q)$ whenever $\mu(E)\le \beta \mu(Q)$.
     \item [{\rm (iv)}]
      There exists $\delta>0$ so that uniformly for all cubes $Q$ and all subsets $E\subset Q$, we have
      $\frac{\sigma(E)}{\sigma(Q)}\lesssim \left(\frac{\mu(E)}{\mu(Q)}\right)^{\delta}$.
    \item [{\rm (v)}]
      There exists $q>1$ so that uniformly for all cubes $Q$, we have
      the $B_q$ estimate 
      $$
      \left( \barint_Q w^{q} d\mu\right)^{1/q} \lesssim \barint_Q w d\mu.
      $$
  \end{itemize}
\end{prop}

For completeness, we sketch the proof.

\begin{proof}
  The implications $(v)\Rightarrow (iv)\Rightarrow (iii)$ are straightforward, the former using H\"older's inequality.
  Similarly $(i)\Rightarrow (ii)$ is straightforward, using the geometric-
  arithmetic mean inequality. To prove $(ii)\Rightarrow (iii)$, one can first establish
  an estimate $\mu(\sett{x\in Q}{w(x)\le \epsilon_1 w_Q})\le \epsilon_2\mu(Q)$ via
  Chebyshev's inequality applied to $\ln(1+ 1/w(x))$, for small $\epsilon_i$, and $\epsilon_1$ depending on $\epsilon_2$.
  From such estimate (iii), with the $d\mu$ and $d\sigma$ swapped, follows.
  But (iii) is symmetric with respect to $d\mu$ and $d\sigma$, as is readily seen from replacing
  $E$ by $Q\setminus E$.

  In particular we note that if any of the conditions $(i)-(v)$ hold,
  then also $d\sigma$ is a doubling measure.
  The deeper implications $(iii)\Rightarrow (i)$ and $(iii)\Rightarrow (v)$ can be proved using Calder\'on--Zygmund decompositions as in \cite[IV.2, Lem. 2.5]{GCRF}.
  For $(iii)\Rightarrow (v)$, replace $dx$ and $w$ by $d\mu$ and $w$.
  For $(iii)\Rightarrow (i)$, replace $dx$ and $w$ by $d\sigma$ and $1/w$.
\end{proof}

We end this section by noting the relations between the classes of weights, which are most important to us.

\begin{cor}   \label{cor:weightrelations}
Let $W$ be a matrix weight and let $a\in\R^n\setminus\{0\}$ be a non-zero vector.
Define the scalar weight $w(x):= |W(x)a|$.
Then 
$$
  W^2\in A_\infty^N(d\mu)\Leftrightarrow W\in A_\infty^N(d\mu)\cap B_2^N(d\mu)
  \Rightarrow w\in B_2(d\mu)\Leftrightarrow w^2\in A_\infty(d\mu).
$$
\end{cor}

\begin{proof}
  The first equivalence is clear from the determinant characterisations of $A_\infty^N(d\mu)$ and $B_2^N(d\mu)$.
  The implication follows from Proposition~\ref{prop:mB2}(i) and the triangle inequality.
  The last equivalence follows from Proposition~\ref{prop:scalarAinfty}.
\end{proof}

Note in particular concerning the last equivalence in Corollary~\ref{cor:weightrelations} that
for scalar weights $B_2(d\mu)\subset A_\infty(d\mu)$. 
See Buckley~\cite[Prop. 3.10]{Buck}.
We do not know if this is true for matrix weights, $N\ge 2$.
Note also that it is essential to work with non-dyadic cubes here. For example dyadic $B_2(d\mu)$
is not contained in dyadic $A_\infty(d\mu)$.

\section{Compactness and stopping times}  \label{sec:proof}

In this section we prove Theorem~\ref{thm:main}, using two
stopping time arguments, one on the test functions and one on
the matrix weight.

\begin{lem}   \label{lem:maximizingvec}
  Let $x,y\in\R^m$ be unit vectors, and let $A: \R^m\to \R^{m'}$ be linear operators.
  Then we have the inequality
  $$
    |Ay|\ge \Big( (x,y)-\sqrt 2\sqrt{1-|Ax|/|A|} \Big)|A|.
  $$
\end{lem}

\begin{proof}
  For any $c>0$ we calculate
\begin{multline*}
  |Ay|= |A(cx+(y-cx))|\ge c|Ax|-|A| |y-cx| \\
  = c|Ax|-|A|\sqrt{1+c^2-2c(x,y)})
  \ge c|Ax| - c|A|(1+1/(2c^2)-(x,y)/c) \\
  = -c(|A|-|Ax|)-|A|/(2c)+ (x,y)|A|.
\end{multline*}
  The result follows by setting $c:= \sqrt{|A|/(2(|A|-|Ax|))}$.
\end{proof}

The following proposition roughly says that we can find an $\epsilon$-net of unit vectors
such that any matrix has lower bounds on a compact approximation of the half-space determined
by one of these unit vectors.
Independently, this argument has earlier been found by Hyt\"onen~\cite{Hunpub}.

\begin{prop}    \label{prop:Katonewsectors}
  Let $\epsilon_1>0$ be given, and let $M$ be a finite set of unit vectors in $\R^N$ such that
  for all unit vectors $v_1\in \R^N$ there exists $v_0\in M$ such that $(v_0,v_1)\ge 1-\epsilon_1^4/8$.
  For each $v_0\in M$ define the set
  $$
    S(v_0):= \sett{(t,x)\in \R^{1+n}_+}{|\gamma_t(x)|\le \tfrac 2{\epsilon_1^3}|\gamma_t(x)v| \text{ for all } v\in D(v_0)},
  $$
  where
  $$
    D(v_0):= \sett{v\in \R^N}{(v,v_0)\ge \epsilon_1 \text{ and } |v|\le 1/\epsilon_1}.
  $$
  Then $\bigcup_{v_0\in M} S(v_0)= \R^{1+n}_+$.
\end{prop}

\begin{proof}
  Let $(t,x)\in \R^{1+n}_+$ and consider the operator $\gamma:= \gamma_t(x)$.
  Let $v_1$ be a unit vector such that $|\gamma|= |\gamma v_1|$.
  Pick $v_0\in M$ such that $(v_0,v_1)\ge 1-\epsilon_1^4/8$.
  We need to show that
  $|\gamma_t(x)|\le \tfrac 2{\epsilon_1^3}|\gamma_t(x)v|$ holds for all $v\in D(v_0)$, so that $(t,x)\in S(v_0)$.
  To this end we apply Lemma~\ref{lem:maximizingvec} twice.
  First apply it to $v_1$ and $v_0$, giving
  $$
    |\gamma v_0|\ge (1-\epsilon_1^4/8)|\gamma|.
  $$
  Then apply it to $v_0$ and $v/|v|$, where $v\in D(v_0)$. This yields
  $$
    |\gamma v|/|v|\ge \Big(\epsilon_1^2-\sqrt 2\sqrt{\epsilon_1^4/8}\Big)|\gamma|= \tfrac 12\epsilon_1^2|\gamma|.
  $$
  This completes the proof.
\end{proof}

The first stopping time argument used in the proof of Theorem~\ref{thm:main}, following \cite{AHLMcT},
is the following.

\begin{prop}   \label{prop:Katostop}
  Fix $\epsilon_2>0$ and a unit vector $v_0\in \R^N$. Consider the stopping criterion
  which selects the maximal subcubes $R\in \mD$ of $Q\in \mD$ for which
  either
  $$
    |E_R b_Q^{v_0}|> 1/\epsilon_2
  $$
  or
  $$
    (v_0, W_Q^{-1}W_R E_R b_Q^{v_0} )<\epsilon_2.
  $$
  Then the first generation stopping cubes, denoted $B_1^b(Q)$, satisfies the packing condition
  $$
    \sum_{R\in B_1^b(Q)} \mu(R)\le (1-\delta) \mu(Q)
  $$
  for some $\delta>0$ depending on $\epsilon_2$, if $\epsilon_2$ is small enough.
\end{prop}

\begin{proof}
  Let $B'(Q)$ denote the stopping cubes for which $|E_R b_Q^{v_0}|>1/\epsilon_2$ and let $B''(Q):=B_1^b(Q)\setminus B'(Q)$.
  We have
$$
  \sum_{R\in B'(Q)} \mu(R)\lesssim \epsilon_2^2\sum_{R\in B'(Q)} \int_R |b_Q^{v_0}|^2 d\mu\le C_3\epsilon_2^2 \mu(Q),
$$
since $|E_R b_Q^{v_0}|^2\lesssim \mu(R)^{-1}\int_R |b_Q^{v_0}|^2 d\mu$ by Proposition~\ref{prop:Etbound}.

Writing $G:= Q\setminus \bigcup_{R\in B''(Q)}R$, we get
\begin{multline*}
  1= (v_0, E_Q b_Q^{v_0})= \sum_{R\in B''(Q)} \left(v_0, W_Q^{-1}W_R E_R b_Q^{v_0}\right)\frac {\mu(R)}{\mu(Q)} \\ +
  \left(v_0, (\mu(Q)W_Q)^{-1} \int_G W(x)b_Q^{v_0}(x)d\mu(x)\right) \\
  \le \sum_{R\in B''(Q)} \epsilon_2 \tfrac {\mu(R)}{\mu(Q)}  +
  \int_G (W(x) (\mu(Q)W_Q)^{-1}v_0, b_Q^{v_0}(x)) d\mu(x) \\
  \le \epsilon_2 + C_3\left( \int_G |W(x)(\mu(Q)W_Q)^{-1}v_0|^2 d\mu(x) \right)^{1/2} \mu(Q)^{1/2} \\
  \le \epsilon_2 + C\left(\frac{\mu(G)}{\mu(Q)}\right)^{\delta/2} \left( \int_Q |W(x)(\mu(Q)W_Q)^{-1}v_0|^2 d\mu(x) \right)^{1/2} \mu(Q)^{1/2} \\
  \le \epsilon_2 + C\left(\frac{\mu(G)}{\mu(Q)}\right)^{\delta/2} \left| \int_Q W(x)(\mu(Q)W_Q)^{-1}v_0 d\mu(x) \right| \mu(Q)^{1/2} \\
  = \epsilon_2 + C (\mu(G)/\mu(Q))^{\delta/2}.
\end{multline*}
We have used that $x\mapsto |W(x)(\mu(Q)W_Q)^{-1}v_0|^2$ is an $A_\infty(d\mu)$ weight
by Corollary~\ref{cor:weightrelations}, with $Q$ fixed.

Therefore
$$
   \sum_{R\in B_1^b(Q)} \mu(R)\le (C\epsilon_2^2+1-((1-\epsilon_2)/C)^{2/\delta}) \mu(Q),
$$
and we see that it suffices to choose $\epsilon_2$ small as claimed.
\end{proof}

Next we set up a stopping time argument to control the oscillations of the averages of $W$.

\begin{prop}   \label{prop:coronastop}
  Fix $\epsilon_3>0$. Consider the stopping criterion
  which selects the maximal subcubes $R\in \mD$ of $Q\in \mD$ for which
  $$
    |W_Q^{-1}W_R - I |>\epsilon_3.
  $$
  Then the generations of stopping cubes, denoted $B_*^W(Q)$, satisfy the packing condition
  $$
    \sum_{R\in B_*^W(Q)} \mu(R)\lesssim \epsilon_3^{-2}\mu(Q).
  $$
\end{prop}

\begin{proof}
Clearly it suffices to prove
$$
  \sup_Q \frac 1{\mu(Q)}\sum_{R\in B_*^W(Q)\setminus\{Q\}} |I-W_{R_*}^{-1} W_R|^2 \mu(R)<\infty.
$$
Fix $Q_0\in \mD$.
Throughout this proof $R_*$ denotes the stopping parent relative to $B_*^W(Q_0)$.
We claim that
$$
\sum_{R\in B_*^W(Q_0)\setminus\{Q_0\}} |I-W_{R_*}^{-1} W_R|^2 \mu(R)\lesssim \mu(Q_0),
$$
with the implicit constant independent of $Q_0$.
To prove this we use the stopping criterion that selects the maximal dyadic subcubes $R$ of $Q$,
for which $R\in B_*^W(Q_0)$ and $|W_R^{-1}W_Q|>\lambda$,  for some fixed parameter $\lambda$.
Use superscript $Q_0$ to denote the so obtained stopping cubes and sawtooths.
By Proposition~\ref{prop:volbergstop}, Lemma~\ref{lem:geomet} and Lemma~\ref{lem:JN}, 
choosing $\lambda$ large, it suffices to prove
$$
   \sum_{R\in B_*^W(Q_0), R_*\in G^{Q_0}(S)} |W_{R_*}^{-1}(W_{R_*}- W_R)|^2 \mu(R)\lesssim \mu(S)
$$
for all $S\in B_*^{Q_0}(Q_0)$, since the above stopping criterion is at least as strict as that in
Proposition~\ref{prop:volbergstop}.
From Proposition~\ref{prop:matrixmarting}, we obtain
$$
  \tr \left(W_S^{-1}\sum_{R\in B_*^W(S)\setminus\{S\}}(W_{R_*}-W_{R})^2\mu(R) W_S^{-1} \right)
  \le \tr( W_S^{-1}(W^2)_S \mu(S) W_S^{-1}).
$$
Since $\tr (A^*A)\approx|A^*A|= |A|^2$ for any matrix, this reads
$$
   \sum_{R\in B_*^W(S)\setminus\{S\}} |W_S^{-1}(W_{R_*}-W_{R})|^2 \mu(R)
  \lesssim |(W^2)_S^{1/2} W_S^{-1}|^2 \mu(S)\lesssim \mu(S),
$$
since $W\in B_2^N(d\mu)$.
This shows
\begin{multline*}
   \sum_{R\in B_*^W(Q_0), R_*\in G^{Q_0}(S)} |W_{R_*}^{-1}(W_{R_*}- W_R)|^2 \mu(R) \\
   \le
   \lambda^2 \sum_{R\in B_*^W(Q_0), R_*\in G^{Q_0}(S)} |W_{S}^{-1}(W_{R_*}- W_R)|^2 \mu(R)
   \\
   \le \lambda^2\sum_{R\in B_*^W(S)\setminus\{S\}} |W_{S}^{-1}(W_{R_*}- W_R)|^2 \mu(R) \lesssim \mu(S),
\end{multline*}
which completes the proof.
\end{proof}

\begin{proof}[Proof of Theorem~\ref{thm:main}]
  By Proposition~\ref{prop:Katonewsectors},
we have
$$
\iint_{\widehat Q} |\gamma(t,x)|^2 \frac{dtd\mu(x)}t\le \sum_{v_0\in M} \iint_{\widehat Q} |\gamma(t,x)1_{S(v_0)}|^2 dtd\mu(x),
$$
where for a fixed $\epsilon_1$-net $M$, the sum is finite.
By Lemmas~\ref{lem:JN} and \ref{lem:geomet} and Propositions~\ref{prop:Katostop} and Proposition~\ref{prop:coronastop}, it suffices to prove the estimate
$$
  \iint_{G^W(S_1)\cap G^b(S_2)} |\gamma(t,x)1_{S(v_0)}|^2 \frac{dtd\mu(x)}t \lesssim \mu(S_2),
$$
for $S_1\supset S_2$.
By Proposition~\ref{prop:Katonewsectors}, we have
$$
|\gamma(t,x)1_{S(v_0)}|\le \tfrac 2{\epsilon_1^3} |\gamma(t,x)v|,
$$
for all $v\in D(v_0)$.
We want to use this estimate with $v= E_R b_{S_2}^{v_0}$,
where we have dyadic cubes $R\in G^W(S_1)\cap G^b(S_2)$ and $S_2\subset S_1$.
To verify that $E_R b_{S_2}^{v_0}\in D(v_0)$,
we have from Proposition~\ref{prop:Katostop} the estimates
$$
  |E_R b_{S_2}^{v_0}|\le 1/\epsilon_2
$$
and
$$
  (v_0, W_{S_2}^{-1}W_{R}E_R b_{S_2}^{v_0})\ge \epsilon_2.
$$
Since $R\in G^W(S_1)$, from Proposition~\ref{prop:coronastop} we get the estimate
$$
  |W_{S_1}^{-1}W_R-I|\le \epsilon_3.
$$
Moreover, we may assume $G^W(S_1)\cap G^b(S_2)\ne \emptyset$, so $S_2\in G^W(S_1)$
and therefore
$$
  |W_{S_1}^{-1}W_{S_2}-I|\le \epsilon_3.
$$
It follows that $|W_{S_2}^{-1}W_{R}-I|\le 2\epsilon_3/(1-\epsilon_3)$.
This gives
$$
  (v_0, E_R b_{S_2}^{v_0})\ge \epsilon_2-\frac{2\epsilon_3}{1-\epsilon_3}\frac 1{\epsilon_2}.
$$
We now fix the values of $\epsilon_1, \epsilon_2$ and $\epsilon_3$ as follows,
depending on the values $C_1, C_2, C_3, C_4, n$ and $N$ in the statement of Theorem~\ref{thm:main}.
First choose $\epsilon_2$ small enough to satisfy the hypothesis in Propositions~\ref{prop:Katostop}.
Then choose $\epsilon_3$ small enough so that
$$
(v_0, E_R b_{S_2}^{v_0})\ge \epsilon_2/2.
$$
Finally choose $\epsilon_1$ small enough so that $E_R b_{S_2}^{v_0}\in D(v_0)$, 
for all $R\in G^W(S_1)\cap G^b(S_2)$.
We obtain the estimate
\begin{multline*}
  \iint_{G^W(S_1)\cap G^b(S_2)} |\gamma(t,x)1_{S(v_0)}|^2 \frac{dtd\mu(x)}t \\ \lesssim
  \iint_{G^W(S_1)\cap G^b(S_2)} |\gamma(t,x) E_R b_{S_2}^{v_0}|^2 \frac{dtd\mu(x)}t \\ \le
  \iint_{\widehat S_2} |\gamma(t,x) E_R b_{S_2}^{v_0}|^2 \frac{dtd\mu(x)}t \le C_4^2 \mu(S_2),
\end{multline*}
which completes the proof.
\end{proof}

\section{A reverse reverse triangle inequality}  \label{sec:reverev}

As an appendix, we include here a result in linear algebra, which was found in a failed
attempt to prove Proposition~\ref{prop:coronastop} by reduction to scalar weights,
but which should have independent interest.

\begin{prop}
  Fix dimension $m$ and consider positive linear operators $A, B: \R^m\to \R^m$.
  Let $\epsilon>0$. Then there exists $\delta>0$ such that
$$
  |Ax-Bx|\le \epsilon |Bx|\qquad\text{for all } x\in \R^m,
$$
for any positive linear operators $A$ and $B$ on $\R^m$ which satisfy
$$
  \big| |Ax|-|Bx| \big|\le \delta |Bx|\qquad\text{for all } x\in \R^m.
$$
\end{prop}

We remark that the more precise dependence of $\delta$ on $m$ as well as on $\epsilon$
remains to be found. Note that in dimension one the result is trivial, and also that
the converse implication is true in any dimension, since
$$
  \big| |Ax|-|Bx| \big|\le |Ax-Bx|
$$
by the reverse triangle inequality.

\begin{proof}
  The proof is by induction on dimension $m$, the details of which we return to towards the end of the proof.
  Consider dimension $1+m$ and split $\R^{1+m}=\R\oplus \R^m$ and the matrix $A$ as
  $$
    A= \begin{bmatrix} a & b^t \\ b & D \end{bmatrix},
  $$
  where $0<a\in \R$, $b\in \R^m$ and $D$ is a positive $m\times m$ matrix.
  By the spectral theorem we may, and will, assume that $B$ is a diagonal matrix.
  Its block decomposition we write
  $$
    B= \begin{bmatrix} e & 0 \\ 0 & F \end{bmatrix},
  $$
  where without loss of generality we assume that $e>0$ is less or equal to the smallest
  of the diagonal elements in the diagonal matrix $F$.

  The positivity of $A$ shows that
  $ax^2+ 2x(b,y)+(Dy,y)\ge 0$
  for all $x\in\R$ and $y\in\R^m$,
  or equivalently
  \begin{equation}  \label{eq:by}
    |(b,y)|^2\le a (Dy,y)\qquad\text{for all } y\in \R^m.
  \end{equation}
  From the hypothesis, we also have $(1-\delta)^2 B^2\le A^2\le (1+\delta)^2 B^2$,
  which in block form becomes
  $$
    (1-\delta)^2\begin{bmatrix} e^2 & 0 \\ 0 & F^2 \end{bmatrix}\le
     \begin{bmatrix} a^2+|b|^2 & ab^t+b^t D \\ ab+Db & D^2+ bb^t \end{bmatrix}
     \le (1+\delta)^2\begin{bmatrix} e^2 & 0 \\ 0 & F^2 \end{bmatrix}.
  $$
  In particular, from the diagonal blocks we have
  \begin{align}
    |a^2+|b|^2-e^2| &\lesssim \delta e^2 \label{eq:b2y} \\
   \big||Dy|^2+|(b,y)|^2-|Fy|^2\big| &\lesssim \delta |Fy|^2,\qquad\text{for all } y\in \R^m. \label{eq:b3y}
 \end{align}
  Similar to the derivation of \eqref{eq:by}, the inequality $(1-\delta)^2 B^2\le A^2$
  shows that
  \begin{multline}   \label{eq:b4y}
    |(ab+Db,y)|^2 \le (a^2+|b|^2-(1-\delta)^2 e^2)(|Dy|^2+|(b,y)|^2-(1-\delta)^2|Fy|^2) \\
    \lesssim \delta^2 e^2|Fy|^2.
  \end{multline}

  We want to show a lower bound $|Dy|\ge \alpha |Fy|$, for some $\alpha>0$. To this end we argue by contradiction
  and assume that there exists $y\in \R^m$ such that $|Dy|< \alpha |Fy|$.
  From \eqref{eq:by} and \eqref{eq:b2y} we see that
  $$
    |(b,y)|^2\lesssim e(\alpha|Fy||y|)\lesssim \alpha|Fy|(e|y|)\lesssim\alpha |Fy|^2.
  $$
  From \eqref{eq:b3y} we get
  $$
    (1-\alpha^2-C\alpha)|Fy|^2\lesssim \delta|Fy|^2,
  $$
  which is a contradiction if $\alpha$ is small, if $\delta<1/2$.
  This shows that $|Dy|\gtrsim |Fy|$ for all $y\in \R^m$.

  From \eqref{eq:b4y} we get, with $y= D^{-1}z$, the
  estimate
  $$
    |((I+aD^{-1})b,z)|^2\lesssim \delta^2e^2|z|^2.
  $$
  This together with the positivity of $aD^{-1}$ shows
  \begin{equation}   \label{eq:b45y}
    |b|\le |(I+aD^{-1})b|\lesssim \delta e.
  \end{equation}
  This inserted into \eqref{eq:b2y} and \eqref{eq:b3y} gives estimates
  \begin{align}
    |a-e| &\lesssim \delta e \label{eq:b5y} \\
   \big||Dy|-|Fy|\big| &\lesssim \delta |Fy|,\qquad\text{for all } y\in \R^m. \label{eq:b6y}
 \end{align}

  We now argue by induction over dimension $m$. If $m=1$, then
  \eqref{eq:b45y}, \eqref{eq:b5y} and \eqref{eq:b6y} show that $|AB^{-1}-I|\le\epsilon$
  provided $\delta$ is small enough.
  If $m\ge 2$, then assuming the proposition proved for $m-1$, it follows
  from \eqref{eq:b6y} that $|DF^{-1}-I|\le \epsilon$ if $\delta$ is small enough.
  Together with \eqref{eq:b45y} and \eqref{eq:b5y}, this shows that $|AB^{-1}-I|\le\epsilon$
  and completes the proof.
\end{proof}

\bibliographystyle{acm}
%GATHER{AKMcDirac.bib}  % makes sure WinEdt finds citations...
%\bibliography{parastopping}

\end{document}